\documentclass{amsart}
\usepackage{arxiv}

\usepackage[utf8]{inputenc} 
\usepackage[T1]{fontenc}    
\usepackage{hyperref}       
\usepackage{url}            
\usepackage{booktabs}       
\usepackage{amsfonts}       
\usepackage{nicefrac}       
\usepackage{microtype}      
\usepackage{lipsum}		

\usepackage{amsmath}%
\usepackage{amsfonts}%
\usepackage{amssymb}%
\usepackage{graphicx}

\def\PP{{\mathcal P}}
\def\QQ{{\mathcal Q}}

\def\CC{{\mathcal C}}

\newtheorem{theorem}{Theorem}
\theoremstyle{plain}

\newtheorem{corollary}{Corollary}

\newtheorem{lemma}{Lemma}

\newtheorem{proposition}{Proposition}

\begin{document}
\pdfoutput=1
\title{The Edge-connectivity of Token Graphs}
\Large

\maketitle

\begin{center}
{\textbf{J. Lea\~nos $^1$ \hspace{1cm}
        M. K. Christophe Ndjatchi$^{2}$}}

\end{center}
\author{
}

\footnote{  Unidad Acad\'emica de Matem\'aticas, Universidad Aut\'onoma de Zacatecas, M\'exico.
             Email: jleanos@matematicas.reduaz.mx       \\

             $^{2}$Department of Physics and Mathematics, Instituto Polit\'ecnico Nacional, UPIIZ,
  P.C. 098160, Zacatecas, M\'exico. Email: mndjatchi@ipn.mx }

\begin{abstract}
Let $G$ be a simple graph of order $n\geq 2$ and let $k\in \{1,\ldots ,n-1\}$. The $k$-token graph $F_k(G)$ of $G$ is the graph whose vertices are the $k$-subsets of $V(G)$, where two vertices are adjacent in $F_k(G)$ whenever their symmetric difference is an edge of $G$. In 2018 J. Lea\~nos and A. L. Trujillo-Negrete proved that if $G$ is $t$-connected and $t\geq k$, then  $F_k(G)$ is at least $k(t-k+1)$-connected. In this paper we show that such a lower bound remains true in the context of edge-connectivity. Specifically, we show that if $G$ is $t$-edge-connected and $t\geq k$, then  $F_k(G)$ is at least
$k(t-k+1)$-edge-connected. We also provide some families of graphs attaining this bound.
\end{abstract}
\keywords{ Token graphs \and edge-connectivity \and Menger's theorem}

{\it AMS Subject Classification Numbers:} 05C40

\section{Introduction}
\label{intro}

Throughout this paper, $G$ is a simple finite graph of order $n\geq 2$ and $k\in \{1,\ldots ,n-1\}$. The $k$-\textit{token graph} $F_k(G)$ of $G$ is the graph whose vertices are all the $k$-subsets of $V(G)$ and two
$k$-subsets are adjacent whenever their symme\-tric difference is a pair of adjacent vertices in $G$. In particular, note that
$F_k(G)$ is isomorphic to $F_{n-k}(G)$. In this paper we often write token graph instead of $k$-token graph. See Figure \ref{fig:token} for an example.\\

\begin{figure}[ht]
\begin{center}
\includegraphics[width=0.8\textwidth]{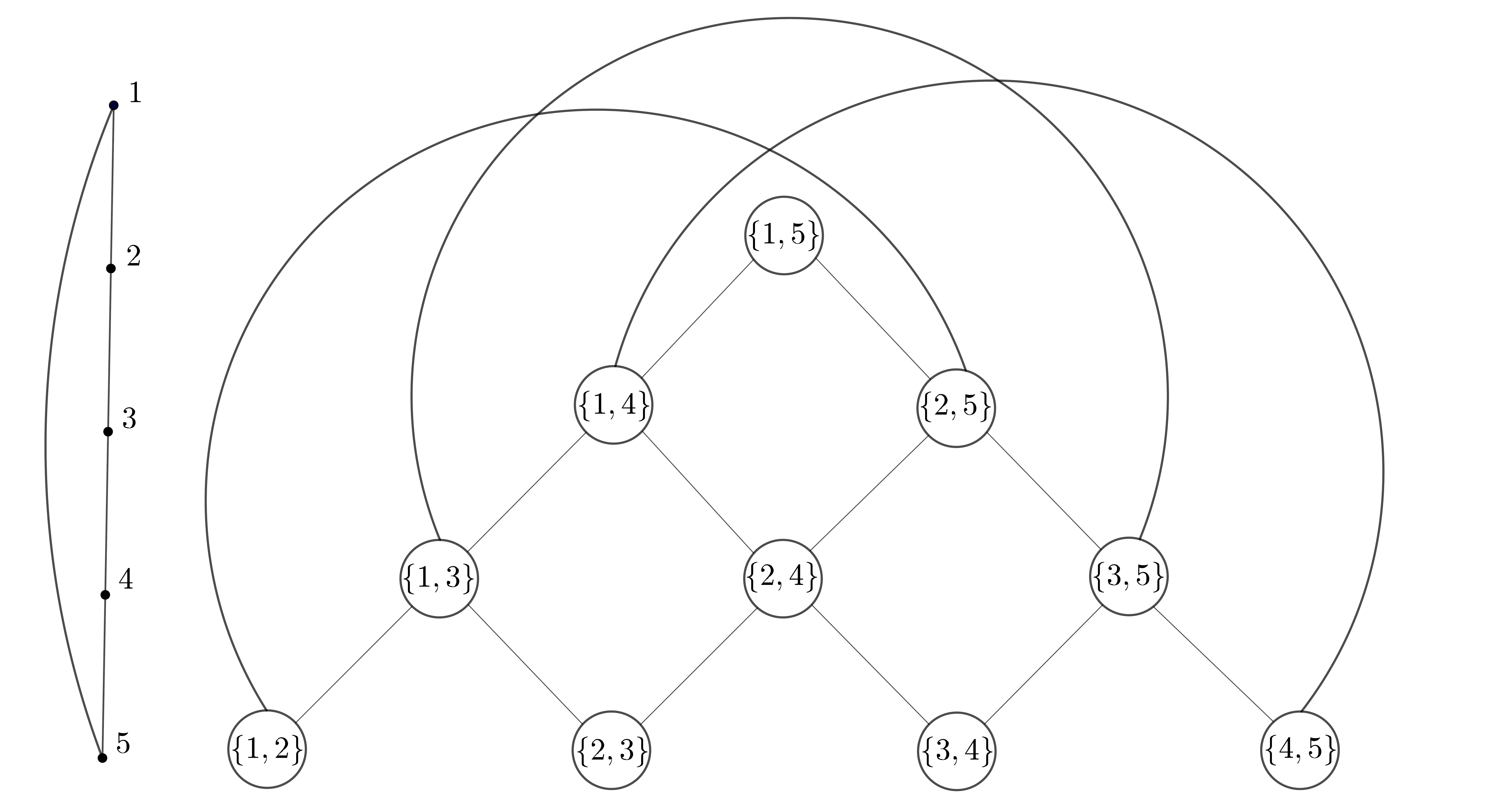}
\caption{\small The graph on the left is the cycle $C_5$, and the graph on the right is $F_2(C_5)$. For instance, note that the symmetric difference of the vertices $\{1,2\}$ and $\{2,3\}$ of $F_2(C_5)$ is $\{1,3\}$, but they are not adjacent in $F_2(C_5)$ because $1$ and $3$ are not adjacent in
$C_5$.}
\label{fig:token}
\end{center}
\end{figure}

As far as we know, the notion of token graphs was introduced by Y. Alavi, M. Behzad, P. Erd{\H{o}}s, and D. R. Lick in 1991 \cite{alavi2}, where the $2$-token graphs were called {\em double vertex graphs}. The next time that this concept appeared in the literature was in 2002, when T. Rudolph in its arXiv preprint ``Cons\-tructing physically intuitive graph invariants" used $F_2(G)$ to study the graph isomorphism problem \cite{rudolph}.  In such a work were exhibited some examples of non-isomorphic graphs $G$ and $H$ which are cospectral, but with $F_2(G)$ and $F_2(H)$ non-cospectral. He emphasized that fact by making the following remark about the eigenvalues of $F_2(G)$:  ``{\em then the eigenvalues of this larger matrix are a graph invariant, and in fact are a more powerful invariant than those of the original matrix $G$}".

In 2007 \cite{aude}, the notion of token graphs was extended by K. Audenaert, C. Godsil, G. Royle, and T. Rudolph  to any integer $k\in\{1,\ldots, n-1\}$, and in such a paper $F_k(G)$ was called the {\em symmetric $k$-$th$ power of $G$}. It was proved in \cite{aude} that the $2$-token graphs of strongly regular graphs with the same parameters are cospectral. In addition, some connections with generic exchange Hamiltonians in quantum mechanics were also discussed. Following Rudolph's study,  A. Barghi, and I. Ponomarenko \cite{barghi}; and A. Alzaga, R. Iglesias, and  R. Pignol  \cite{alzaga} proved, independently, that for a given positive integer $m$ there exists infinitely many pairs of non-isomorphic graphs with cospectral $m$-token graphs.

In 2012 \cite{FFHH}, R. Fabila-Monroy, D. Flores-Pe{\~n}aloza, C. Huemer, F. Hurtado, J. Urrutia,  and D. R. Wood   reintroduced, independently, the concept of $k$-token graphs as ``{\em a model in which $k$ indistinguishable tokens move from vertex to vertex along the edges of a graph}".  From this new way of defining $F_k(G)$ is not hard to see that the study of the structure of $F_k(G)$ is intimately related to the class of problems known as {\em reconfiguration problems}. For more on reconfiguration problems, see for example \cite{auletta,caline,demaine,ratner,miller,vanden,yama}.
On the other hand, most of the results reported in \cite{FFHH} address the following classical appro\-ach: Given a graph $G$ and a graph invariant $\eta$; What can we say about $\eta(F_k(G))$, when $\eta(G)$ is known? In  \cite{FFHH} were studied the cases in which the graph invariant is the (vertex-) conne\-ctivity, the diameter, the clique number, the chromatic number, etc. This approach has been continued  in several investigations (see, e.g., \cite{alba,token2,gomez,leanos,mira}).

The $k$-token graphs also are a generalization of Johnson graphs: if $G$ is the complete graph of order $n$, then $F_k(G)$ is isomorphic to the
Johnson graph $J(n,k)$. The Johnson graphs have been widely studied and the analysis of many of its combinatorial properties is an active area of research~\cite{alavi,brouwer,riyono,etzion,terwilliger}.

Concerning the connectivity of token graphs, in \cite{FFHH} were given some families of graphs of order $n$ with connectivity exactly $t$, and whose
$k$-token graphs have connectivity exactly $k(t-k+1)$, whenever $k \leq t$. In that paper was conjectured that it is best possible. More formally, they conjectured that if $G$ is $t$-connected and $k \leq t$, then $F_k(G)$ is at least  $k(t-k+1)$-connected. In 2018 \cite{leanos}, J. Lea\~nos and A. L. Trujillo-Negrete proved this conjecture. In the present paper we are interested in the edge-connectivity  of $F_k(G)$.

As usual, if $H$ is a graph, then we shall use $\kappa(H)$, $\lambda(H)$, and $\delta(H)$ to denote, respectively, the {\em connectivity, edge-connectivity,} and  {\em minimum degree} of $H$.  It is well known that any connected graph $H$ satisfies

 \begin{equation} \label{eq1}
\delta(H)\geq \lambda(H)\geq \kappa(H).
\end{equation}

  On the other hand, for any of these two inequalities, it is easy to find simple connected graphs whose value of the right side parameter is close to $1$ while the value of the left side parameter may be significantly longer. See Figure \ref{fig:intro1} for a couple of examples.
\begin{figure}[ht]
\begin{center}
\includegraphics[width=0.8\textwidth]{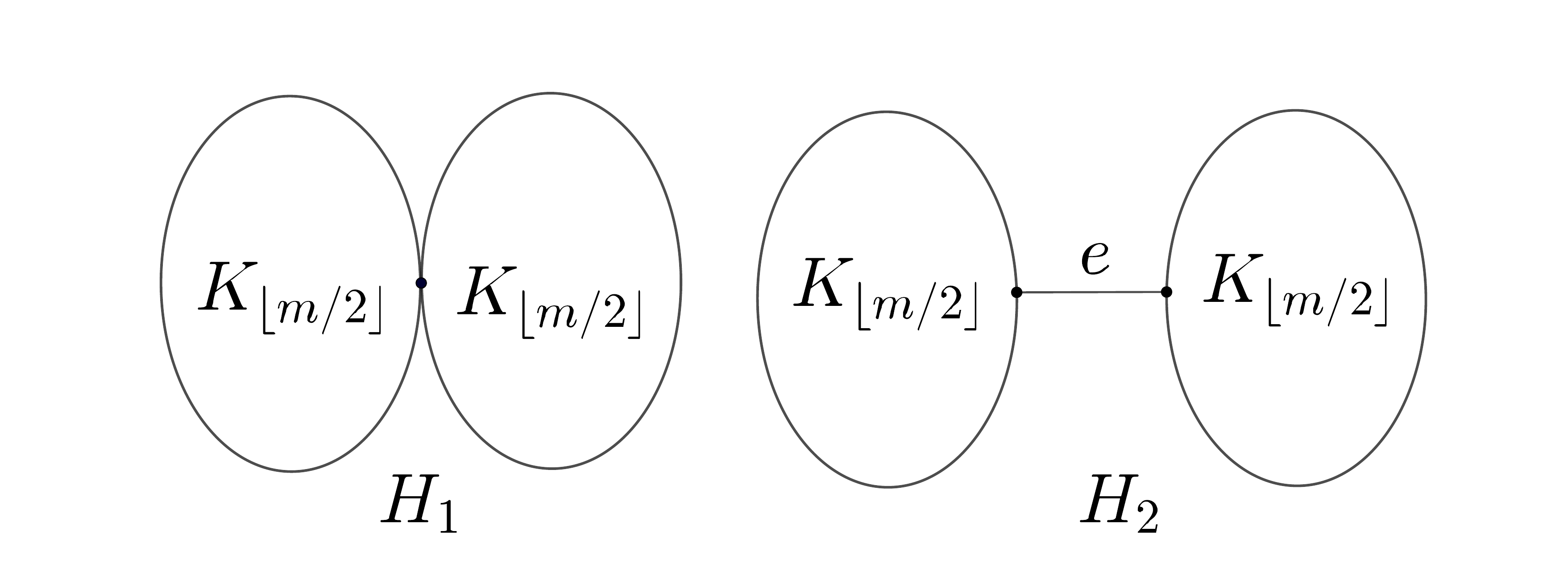}
\caption{\small Let $m\geq 2$ be a large integer and let $r:={\lfloor m/2 \rfloor}$.
On the left we have the graph $H_1$ constructed from two copies of $K_r$ by identifying a vertex of a copy with a vertex of the other copy. Then
$H_1$ has order $2r-1$ and $r-1=\lambda(H_1)\geq \kappa(H_1)=1$.
 The graph on the right is $H_2$ constructed by connecting two copies of $K_r$ by means of a new edge $e$. Then $H_2$ has order $2r$
and $r-1=\delta(H_2)\geq \lambda(H_2)=1$.
}
\label{fig:intro1}
\end{center}
\end{figure}
Our goal in this paper is to study the edge-connectivity of $F_k(G)$ for the case in which $k\leq \lambda(G)$, and
in Theorem \ref{thm:main} we report our main result, which, as we shall see in Section \ref{sec:Examples}, is best possible.

On the other hand, since at the begining of this work we were aware  of the following lower bound for $\lambda(F_k(G))$
due to \cite{leanos} and (\ref{eq1}):

\begin{equation} \label{eq2}
\lambda(F_k(G))  \geq \kappa(F_k(G)) \geq k(\kappa(G)-k+1).
\end{equation}

Since the difference between $\kappa(G)$ and $\lambda(G)$ can be
 significantly large, we remark that the bound of $\lambda(F_k(G))$ provided by Theorem \ref{thm:main} can be significantly better than those in
 (\ref{eq2}). In particular, note that if $G$ is connected, then the bound in Theorem \ref{thm:main} depends only on the edge-connectivity of
 $G$ but not on the (vertex-) connectivity of $G$, which makes the study of $\lambda(F_k(G))$ interesting in its own right.

\begin{theorem}\label{thm:main} Let $G$ be a connected graph of
order $n$, and let $k$ be a positive integer such that $k \leq n-1$ and
$k \leq \lambda(G)$. Then $\lambda(F_k(G))\geq k(\lambda(G)-k+1)$.
\end{theorem}

Next corollary is a direct consequence of Inequality \ref{eq1} and Theorem \ref{thm:main}.

\begin{corollary}\label{thm:example}
 Let $G$ be a connected graph of order $n$, and let $k$ be a positive integer such that $k \leq n-1$ and
$k \leq \lambda(G)$. If $\delta(F_k(G))= k(\lambda(G)-k+1)$ then $\lambda(F_k(G))=\delta(F_k(G))$.
\end{corollary}

Our approach in this paper is a refinement of some ideas and techniques introduced in \cite{FFHH} and  \cite{leanos}. As we shall see in
Sections  \ref{sec:auxililar} and \ref{sec:proof}, such a combination has resulted  in an improvement of our understanding over the structure of the paths of $F_k(G)$. In particular, we remark that the proof that we present here is significantly shorter than those in \cite{leanos}.

The rest of this paper is organized as follows. In Section \ref{sec:auxililar} we give some tools and well known facts which will be used in the proof of our main result. The proof of Theorem \ref{thm:main} is given in Section \ref{sec:proof}. Finally, in Section \ref{sec:Examples} we shall provide two families of graphs which
satisfy the hypothesis of Corollary \ref{thm:example}, and hence attain the lower bound in Theorem \ref{thm:main}. Additionally, we also give an example that  shows that the hypothesis $\lambda(G)\geq k$ in Corollary \ref{thm:example} is necessary.

\section{Paths of $F_k(G)$ that comes from the concatenation of paths of $G$}\label{sec:auxililar}

In this paper we use the definition of $F_k(G)$ given in \cite{FFHH}. Hence, we  think of a vertex of $F_k(G)$ as a configuration of $k$ tokens placed at
$k$ distinct vertices of $G$, where two of such configurations are adjacent in $F_k(G)$ if and only if one configuration can be reached from the other
 by moving one token along an edge of $G$ from its current vertex to an unoccupied vertex.

Our aim in this section is to present some basic facts which will be used in the next section. In particular, using a technique introduced in
 \cite{FFHH}, we explain how to produce paths in $F_k(G)$ by moving tokens along certain paths of $G$. As we will see, the existence of such paths in $F_k(G)$ and Menger's Theorem for edge-connectivity are the heart of the proof of Theorem \ref{thm:main}. We start by giving some concepts and notations.

 Let $H$ be a simple connected graph. We recall that if $u$ and $v$ are vertices of $H$, then a $u-v$ {\em path} is a path of $H$ with ends $u$ and $v$.
Let $P$ be a path of $H$ of length at least one. If $u$ and $v$ are vertices of $P$, then we shall use $u\overrightarrow{P}v$ to denote the only directed subpath of $P$ which goes from $u$ to $v$. If $P=uv$, then we write $\overrightarrow{uv}$ instead of  $u\overrightarrow{P}v$.

Now we give some notation introduced in \cite{FFHH} with slight adaptations.
Let $X$ be a $k$-set of $V(G)$. If $P$ is an $x-y$ path in $G$ directed from $x$ to $y$ such that $x\in X$ and $y\notin X$, then we say that $P$ is an
{\em admissible} path for $X$.

Let $P$ be an $x-y$ path of $G$ admissible for $X$, and let
$Y:=(X\setminus \{x\})\cup \{y\}$. Clearly, $X$ and $Y$ are distinct vertices of $F_k(G)$. Now  suppose that $X\cap P=\{v_1,\ldots ,v_q\}$ ordered by
$P$ (although not necessarily consecutive in $P$), where $v_1 = x$. We shall use $X\underset{xPy}\longrightarrow Y$ to denote the $X-Y$ path in $F_k(G)$
directed from $X$ to $Y$ corresponding to the following sequence of token moves: First move the token at $v_q$ along $P$ to $y$, then for
$i =q-1,q-2,\ldots,1$ move the token at $v_i$ along $P$ to $v_{i+1}$.

For an example, let us consider the cycle $C_5$ and $F_2(C_5)$ in Figure \ref{fig:intro1}. Let $Q$ be the $2-3$-path of $C_5$ defined as
$Q:=2,1,5,4,3$. Clearly, $Q$ is admissible  for $X=\{2,4\}$. Then
$x=2, y=3, Y=\{3,4\}, q=2, v_2=4$, and $v_1=2$. Then,
$$\QQ:=X\underset{xQy}\longrightarrow Y=\{2,4\} \underset{4,3}\longrightarrow \{2,3\} \underset{2,1}\longrightarrow \{1,3\}  \underset{1,5}\longrightarrow \{5,3\}  \underset{5,4}\longrightarrow \{4,3\}.$$

Note that when the token at $v_i$ goes from $v_i$ to  $v_{i+1}$, there are no other tokens in $v_iPv_{i+1}$.
 Thus these moves correspond to a path in $F_k(G)$, which starts in $X$ and ends in $Y$. Also note that each edge $X_1,X_2$ in
 $X\underset{xPy}\longrightarrow Y$ corresponds to an edge in $P$. More precisely, if $X_1$ and $X_2$ are adjacent vertices in
 $X\underset{xPy}\longrightarrow Y$, then their symmetric difference consists of two adjacent vertices of $P$, say $x_1\in X_1\setminus X_2$ and
 $x_2 \in X_2\setminus X_1$, which are adjacent in $P$ (for instance, the first edge $\{2,4\}, \{2,3\}$ of $\QQ$ corresponds naturally to the edge $4,3$ of $Q$). Then $X\underset{xPy}\longrightarrow Y$ and $P$ have the same length.

Now suppose that $X$ is a $k$-set of $V(G)$, that $P$ is an $x-y_i$ path of $G$ admissible for $X$, and that $Q$ is a $y_j-w$ path of $G$ admissible for  $Y:=(X\setminus \{x\})\cup \{y_i\}$. Let $W:=(Y\setminus \{y_j\})\cup \{w\}$. Then we will denote the concatenation of the paths $X\underset{x P y_i}\longrightarrow Y$ and $Y\underset{y_jP w}\longrightarrow W$ simply by
$$X\underset{x P y_i}\longrightarrow Y\underset{y_j P w}\longrightarrow W.$$ Note that such a concatenation is a directed walk of $F_k(G)$ which starts in $X$ and ends in $W$, where possibly $X=W$. The concatenation of more than 2 of such paths is defined and denoted analogously.

 The next proposition will be useful.

 \begin{proposition}\label{prop:adjacents}
 Let $H$ be a simple connected graph. Then $H$ is $t$-edge-co\-nnected if and only if for each edge $e=uv$ of $H$ there are $t$ edge-disjoint $u-v$ paths.
\end{proposition}

\begin{proof} The forward implication follows immediately from Menger's Theorem for edge-connectivity. Conversely, suppose that for each edge
$uv$ of $H$ there are $t$ edge-disjoint $u-v$ paths, and consider a minimum cut edge $C$ of $H$. Then, it is enough to show that $C$ has at least
$t$ edges. Since $H$ is connected, then $C$ has at least one edge, say $e=xy$. But by hypothesis  $H$ has at least $t$ edge-disjoint $x-y$ paths.
Because $C$ is a cut edge which separates $x$ from $y$, then each of such $x-y$ paths contains at least one edge of $C$, and hence $C$ has at
least $t$ edges.
\end{proof}

We end this section by showing that the lower bound given in Theorem \ref{thm:main} is in fact a lower bound for the minimum degree of $F_k(G)$.

\begin{proposition}\label{prop:min-degree}  Let $k,t,n$ be integers such that $2\leq k\leq t\leq n-1$.
Let $G$ be a $t$-edge-connected simple graph of order $n$. Then $\delta(F_k(G))\geq k(t-k+1)$.
\end{proposition}

\begin{proof}
Let $G, k, n, t$ be as in the statement of Proposition \ref{prop:min-degree}. Let $X$ be a vertex of $F_k(G)$. From the $t$-edge-connectivity of $G$
we know that each vertex $v$ of $G$ has degree at least $t$. Then each $v\in X$ has at least $t-k+1$ neighbors in $G\setminus X$. Let $F$ be the set of edges of $G$ with an end in $X$ and the other end in $G\setminus X$. Then $|F|\geq k(t-k+1)$. From the definition of $F_k(G)$ it is easy to see that there is a one-to-one correspondence between the edges of $F$ and the edges of $F_k(G)$ that are incident with the vertex $X\in V(F_k(G))$. Then $\delta(F_k(G))\geq k(t-k+1)$, as required.
\end{proof}

\section{Proof of Theorem \ref{thm:main}}\label{sec:proof}

In view of Proposition \ref{prop:adjacents} and Menger's Theorem for edge-connectivity, in order to show Theorem \ref{thm:main}, it is enough to prove Lemma \ref{lem:horse}.  For brevity, in the rest of the section, if $m$ is a positive integer, then we use $[m]$ to denote $\{1,\ldots ,m\}$.

\begin{lemma}\label{lem:horse}
Let $G$ be a connected graph of order $n\geq 3$, and let $k$ be a positive integer such that $k \leq n-1$. If $X$ and $Y$ are two adjacent vertices of $F_k(G)$ and $k\leq \lambda(G)$, then  $F_k(G)$ has $k(\lambda(G)-k+1)$ edge-disjoint $X-Y$ paths.
\end{lemma}

\begin{proof} Let $G, k, n, \lambda(G), X$ and $Y$ as in the statement of Lemma \ref{lem:horse}. Since  Lemma \ref{lem:horse} holds trivially for $k = 1$ because $F_1(G)$ and $G$ are isomorphic, we assume that $k\geq 2$. For brevity, let $t:=\lambda(G)$.

Since $X$ and $Y$ are adjacent in $F_k(G)$, then the symmetric difference of $X$ and $Y$ consists precisely of two vertices $x, y\in V(G)$ which are
adjacent in $G$. We assume that $x\in X\setminus Y$, $y\in Y\setminus X$, and that $e$ is the edge of $G$ joining $x$ to $y$.
Then $X=\{x\}\cup Z$ and $Y=\{y\} \cup Z$, where $Z:=\{z_1,\ldots ,z_{k-1}\}$ is a subset of $k-1$ vertices of
$V(G)\setminus \{x,y\}$. See Figure \ref{fig:proof}.

For $i\in[k-1]$, let $u_{i,1}, \ldots u_{i,m_i}$ be the vertices of $G-(X\cup Y)$ that are adjacent to $z_i$ in $G$. Since $G$ is $t$-edge-connected, then
$\delta(G)\geq t$, and so $m_i\geq t-k\geq 0$. For  $i\in [k-1], m_i\geq 1$, and $j\in [m_i]$, let
\begin{align*}\label{first-paths}
Z^x_{i,j} & :=(X\setminus \{z_i\})\cup \{u_{i,j}\}, \\
 Z^y_{i,j} & :=(Y\setminus \{z_i\})\cup \{u_{i,j}\}, and \\
{\PP}_{i,j} & :=X\underset{z_iu_{i,j}}\longrightarrow Z^x_{i,j} \underset{xy}\longrightarrow Z^y_{i,j} \underset{u_{i,j}z_i} \longrightarrow Y.
\end{align*}

Let us define $$\CC_1 :=\{\PP_{i,j} |  i\in [k-1], m_i\geq 1,  \mbox{ and } j\in [m_i]\}.$$

 Note that  $\CC_1$ is a collection of $X-Y$ paths of $F_k(G)$ directed from $X$ to $Y$, all of length $3$, and that $|\CC_1|=m_1+\cdots + m_{k-1}$.

Now we show that the paths in $\CC_1$ are pairwise edge-disjoint.
Let $\PP_{i_1,j_1}$ and $\PP_{i_2,j_2}$  be two distinct paths of $\CC_1$. If $i_1\neq i_2$, then $z_{i_1}$ is the only element of $Z$ that every internal vertex of
 $\PP_{i_1,j_1}$ does not contain. Similarly, $z_{i_2}$ is the only element of $Z$ that every internal vertex of
 $\PP_{i_2,j_2}$ does not contain. Thus $\PP_{i_1,j_1}$ and  $\PP_{i_2,j_2}$ have no internal vertices in common, and hence they are edge-disjoint.  Then we can assume that $i_1=i_2$ and that $j_1\neq j_2$.  Let $i:=i_1=i_2$ and let $r\in \{j_1,j_2\}$. Note that
 if $U$ is any internal vertex of $\PP_{i,r}$, then $U\cap \{u_{i,j_1},u_{i,j_2}\}=\{u_{i,r}\}$. Then $\PP_{i,j_1}$ and  $\PP_{i,j_2}$ have no internal vertices in common.

From the $t$-edge-connectivity of $G$ and Menger's Theorem we know that $G$ has a collection $N=\{P_1,\ldots , P_t\}$ of $t$ pairwise edge-disjoint $x-y$ paths. Let us assume (without loss of generality) that every path in $N$ is directed from $x$ to $y$.

Let $Z_{xy}$ be the subset of $Z$ consisting of all vertices of $G$ that are adjacent to both $x$ and $y$ in $G$. Let $\ell:=|Z_{xy}|$.
Then $\ell\in \{0,\ldots, k-1\}$. For $z\in Z_{xy}$, let $P_{xzy}$ be the $x-y$ path of $G$ directed from
 $x$ to $y$ and defined by $xzy$. Let $N_{xzy}$ be the collection formed by such $\ell$ paths of length 2.
Now we show that by making slight modifications of $N$, if necessary, we can assume that every path of $N_{xzy}$ is in $N$.

Let $z\in Z_{xy}$. We analyze three cases separately.

{\sc Case 1.} Suppose that neither $xz$ nor $zy$ belongs to some path of $N$. Then $P_{xzy}$ is a directed $x-y$ path from $x$ to $y$ that does not belong to $N$. In such a case, we add the $P_{xzy}$ to $N$ and remove any  path $P_r$ of $N\setminus N_{xzy}$ in order to keep $N$ as a collection of $t$ paths. Note that such $P_r$ exists because $t > k-1\geq \ell$.

{\sc Case 2.} Suppose that exactly one of $xz$ or $zy$ belongs to some path of $N$, say $P_r$. Then we add the corresponding $P_{xzy}$ to $N$ and remove $P_r$ from it.

{\sc Case 3.} Suppose that both $xz$ and $zy$ are in paths of $N$. If they are in the same path, then $xzy=P_{xzy}\in N$ and we are done. Then we can assume that  $xz$ and $zy$ lie in different paths of $N$. Suppose that  $xz\in P_i$ and $zy\in P_j$ with $i\neq j$. In this case, note that $xzy$ and
$P'_{xzy}:=P_j\cup P_i\setminus \{xz,zy\}$ are two $x-y$ paths of $G$ directed from $x$ to $y$. As before, we can assume that $xzy$ and $P'_{xzy}$ are paths of $N$ instead of $P_i$ and $P_j$.

It is easy to see that none of the modifications in each of these three cases affects the edge-disjointness property of the $x-y$ paths of $N$.

 By performing a relabeling if necessary, let us assume that
$N=\{P_1,\ldots ,P_t\}$ and that $N\setminus N_{xzy}=\{P_{\ell+1}, \ldots, P_{t}\}$. We let $N_0$ denote the set of all paths of $N$ that are disjoint from
$Z$. For $P_r\in N_1:=N\setminus (N_{xzy} \cup N_0)$, let $x_r$ (respectively, $y_r$) be the first (respectively, last) internal vertex of $P_r$. Then $xx_r$ and
$y_ry$ are the first and last, respectively, edges of $P_r$. Since $P_r\in N_1$, then it has length at least $3$, and so $x_r\neq y_r$. If $P_r\in N_1$ is such that at least one of $x_r$ or $y_r$ belongs to $Z\setminus Z_{xy}$, then we say that $P_r$ is {\em bad}. Let us denote by $N_b$ the set of all bad paths of $N_1$, and let $N_g:=N_1\setminus N_b$. Since each vertex
of $Z\setminus Z_{xy}$ is adjacent to at most one of $x$ or $y$, and the paths of $N_1$ are pairwise edge-disjoint, then  each vertex of  $Z\setminus Z_{xy}$
 is contained in at most one bad path, and so $|N_b|\leq k-1-\ell$.

 We remark that $\{N_{xzy}, N_0, N_b, N_g\}$ is a partition of $N$. Let $n_0:= |N_0|, n_b:=|N_b|,$ and $n_g:=|N_g|$. Then $t=\ell+n_0+n_b+n_g$. Let us assume that the $n_g$ paths of $N$ with the greatest indices are the elements of $N_g$. More formally,  let us assume that $N_g=\{P_{t-n_g+1}, \ldots ,P_{t}\}$.

For every $P_r\in N_0\cup N_{xzy}$, let ${\PP}_{r}:=X\underset{xP_ry}\longrightarrow Y$. Now let us define $$\CC_{2}:=\{\PP_{r} | P_r\in N_0\cup N_{xzy}\}.$$

For every $P_r\in N_g$, let
\begin{align*}
X_r & :=(X\setminus \{x\})\cup \{x_r\}, \\
Y_r & :=(Y\setminus \{y\})\cup \{y_r\}, and\\
{\PP}_{r} & :=X\underset{xx_r}\longrightarrow X_r\underset{x_rP_ry_r}\longrightarrow Y_r \underset{y_ry}\longrightarrow Y.
\end{align*}
Finally, we define  $$\CC_g :=\{\PP_{r} | P_r\in N_g\}.$$

Clearly,  $\CC_{2}\cup \CC_g$ is a collection of $X-Y$ paths of $F_k(G)$ directed from $X$ to $Y$. Now we show
 that if $\PP_i$ and $\PP_j$ are distinct paths of $\CC_{2}\cup \CC_g$, then they are edge-disjoint.
   Suppose that $P_i$ and $P_j$ are the corresponding $x-y$ paths in $N$ of $\PP_i$ and $\PP_j$, respectively.
Seeking a contradiction, suppose that $\PP_i$ and  $\PP_j$ have a common edge, say $AB$, in $F_k(G)$.
The definition of $F_k(G)$ implies that $G$ must have two adjacent vertices $a$ and $b$ such that $A\setminus B=\{a\}$
and $B\setminus A=\{b\}$. Since $G$ is simple, then the only way to reach configuration $B$ from configuration $A$, in one step, is moving a token from $a$ to $b$ along the edge $f=ab$. Then the edge $f$ must belong to both $P_i$ and $P_j$, contradicting that $P_i$ and $P_j$ are edge-disjoint.

 Now we show that if $\PP_r\in \CC_{2}\cup \CC_g$ and $\PP_{i,j}\in \CC_1$ then $\PP_r$ and $\PP_{i,j}$ are edge-disjoint.
 Let $P_r$ be the underlying $x-y$ path of $\PP_r$ in $G$.
We start by noting that each internal vertex of $\PP_{i,j}$ contains exactly one of $x$ and $y$.
On the other hand, for $\PP_r$ we have the following:
\begin{itemize}
\item If $P_r\in N_0$, then neither $x$ nor $y$ belongs to any internal vertex of $\PP_r$.
\item If $P_r\in N_{xyz}$, then both $x, y$ belong to the unique internal vertex of $\PP_r$.
\item If $P_r\in N_{g}$, then neither $x$ nor $y$ belongs to any internal vertex of $\PP_r$.
\end{itemize}
These facts imply that  $\PP_r$ and $\PP_{i,j}$ have no internal vertices in common, and hence they are edge-disjoint.

In summary, we have proved that $\CC_1\cup \CC_2 \cup \CC_g$ is a collection of pairwise edge-disjoint $X-Y$ paths of $F_k(G)$ of cardinality
$\lambda:=m_1+\cdots +m_{k-1}+\ell+n_0+n_g$.

It remains to show that $\lambda\geq k(t-k+1)$.
The $t$-edge-connectivity of $G$ implies that $\delta(G)\geq t$. From the definition of $Z_{xy}$ we know
 that each $z\in Z\setminus Z_{xy}=\{z_{\ell+1}, \ldots , z_{k-1}\}$ is adjacent to at most one of $x$ or $y$. This implies that
$m_i\geq t-(k-1)$ whenever $i\in \{\ell+1, \ldots , k-1\}$.
Similarly, for $i\in \{1,\ldots ,\ell\}$ we have that $m_i\geq t-k$. Then, $$m_1+\cdots +m_{k-1}\geq \ell(t-k)+(k-1-\ell)(t-k+1)=(k-1)(t-k+1)-\ell.$$

On the other hand, by combining $t=\ell+n_0+n_b+n_g$ and $n_b\leq k-1-\ell$, we obtain $t-k+1\leq n_0+n_g$. Then,
$$\lambda\geq (k-1)(t-k+1)-\ell + \ell +t-k+1=k(t-k+1).$$
\end{proof}

\begin{figure}[ht]
\begin{center}
\includegraphics[width=0.6\textwidth]{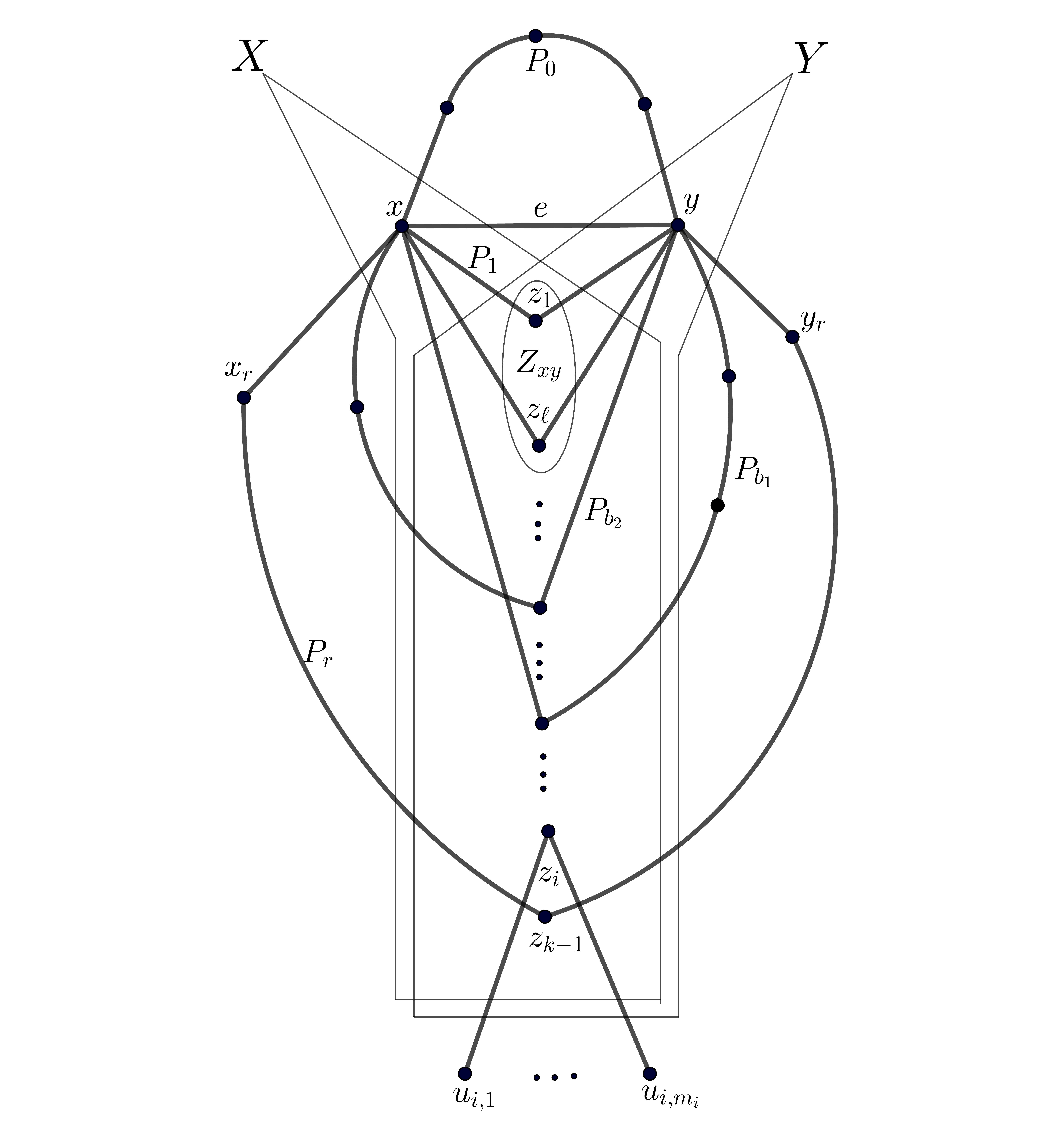}
\caption{\small Here we shown some of the subsets of $V(G)$ involved in the proof of Lemma \ref{lem:horse}. Every thick straight line segment represents an edge of $G$, and every thick convex arc is a path of $G$. Note that $P_0, P_1=xz_1y, P_{b_1}, P_{b_2},$ and $P_r$ form a set of $x-y$ pairwise edge-disjoint paths of $G$. The paths $P_{b_1}$
and $P_{b_2}$ are bad paths of $N_1$. Similarly, note that $P_0\in N_0, P_1\in N_{xzy},$ and $P_r\in N_g$.
}
\label{fig:proof}
\end{center}
\end{figure}


\section{Theorem \ref{thm:main} is best possible}\label{sec:Examples}

A couple of concrete examples satisfying the hypothesis of Corollary \ref{thm:example} are
the complete graph $G:=K_{t+1}$ and the graph $G:=K_t \equiv K_t$ that results by connecting  two copies of $K_t$ by means of a matching of size $t$.
Hence,  $\lambda(F_k(G))=k(t-k+1)$ for $2\leq k \leq t$ and $G\in \{K_{t+1}, K_t \equiv K_t\}$.

Finally, note that the graph $G$ on the left in Figure \ref{fig:graphG} has $\lambda(G)=1$ and $\delta(G)=3$, whereas its $2$-token graph $F_2(G)$ (on the right) has $\lambda(F_2(G))=3$ and $\delta(F_2(G))=4$. The last two values disagree with the conclusion of Corollary \ref{thm:main}. Such a discrepancy reveals that the hypothesis $\lambda(G)\geq k$ in Corollary \ref{thm:main} is necessary.

\begin{figure}
\begin{center}
\includegraphics[width=0.9\textwidth]{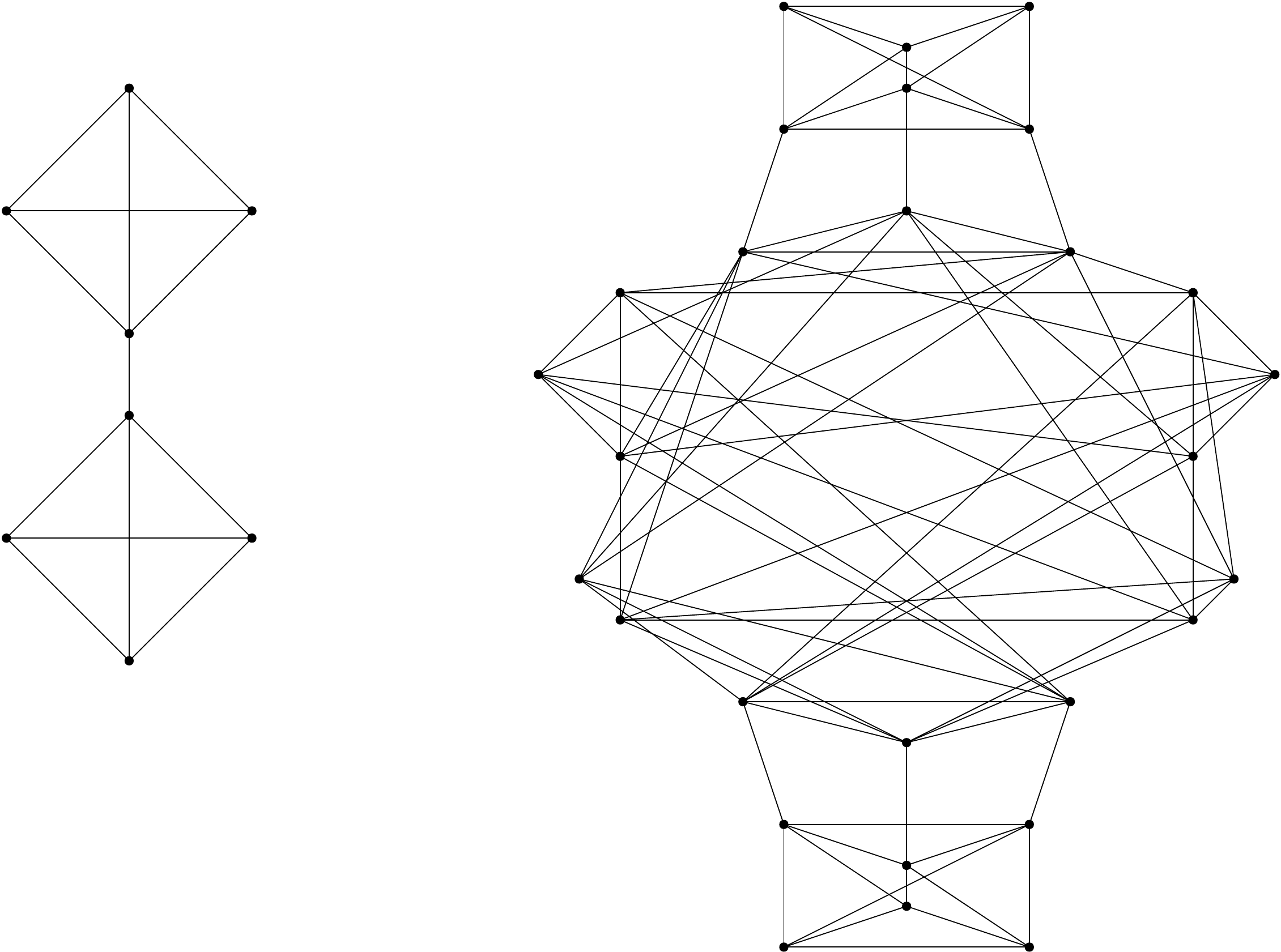}
\caption{\small  The graph on the right $F_2(G)$ is the $2$-token graph of the graph $G$ on the left. They show that the hypothesis $\lambda(G)\geq k$ in  Corollary \ref{thm:example} is necessary. Note that $\lambda(G)=1, \delta(G)=3,\lambda(F_2(G))=3$, and $\delta(F_2(G))=4$.}
\label{fig:graphG}
\end{center}
\end{figure}

\section*{Acknowledgements}
We thank A. L. Trujillo-Negrete for helpful discussions and by providing the example in Figure \ref{fig:graphG}.





\end{document}